\documentclass{amsproc}

\usepackage{graphicx}
\usepackage{mathrsfs}

\usepackage{hyperref}

\newtheorem{theorem}{Theorem}[section]
\newtheorem{lemma}[theorem]{Lemma}
\newtheorem{prop}[theorem]{Prop}

\theoremstyle{definition}
\newtheorem{definition}[theorem]{Definition}
\newtheorem{example}[theorem]{Example}

\theoremstyle{remark}
\newtheorem{remark}[theorem]{Remark}
\newtheorem{question}[theorem]{Question}

\numberwithin{equation}{section}

\newcommand{\C}{\mathbf{C}}
\newcommand{\Z}{\mathbf{Z}}
\newcommand{\Q}{\mathbf{Q}}
\newcommand{\F}{\mathbf{F}}

\newcommand{\SLC}{\operatorname{SL}(2, \C)}
\newcommand{\SL}{{\operatorname{SL}(2, \R)}}

\newcommand{\PSLC}{{\operatorname{PSL}(2,\C)}}
\newcommand{\PSLZ}{{\operatorname{PSL}(2,\Z)}}

\newcommand{\PGL}{{\operatorname{PGL}(2,\R)}}
\newcommand{\R}{\mathbf{R}}
\newcommand{\PSL}{{\operatorname{PSL}(2,\R)}}
\newcommand{\SO}{{\operatorname{SO}(2,\R)}}

\newcommand{\SU}{{\operatorname{SU}(2,\C)}}

\newcommand{\Dev}{{\operatorname{Dev}}}

\newcommand{\Mod}{{\operatorname{Mod}}}
\newcommand{\Homeo}{{\operatorname{Homeo}}}

\newcommand{\Hom}{{\operatorname{Hom}}}

\newcommand{\Aut}{{\operatorname{Aut}}}

\newcommand{\Int}{{\operatorname{Int}}}
\newcommand{\Sym}{{\operatorname{Sym}}}

\newcommand{\Tr}{{\operatorname{Tr} }}

\newcommand{\Id}{{\operatorname{Id}}}

\newcommand{\Isom}{\operatorname{{Isom}}}

\newcommand{\CPS}{{\mathbf {C}P^1}}
\newcommand{\CS}{{\mathbf{C}}}
\newcommand{\HS}{{\mathbb{H}}}

\newcommand{\RPS}{{\mathbf{R}P^1}}

\begin{document}

\title{A discreteness algorithm for 4-punctured sphere groups}

\author{Caleb Ashley}
\address{Department of Mathematics, University of Michigan, Ann Arbor, Michigan , 48109}
\email{cjashley@umich.edu}



\dedicatory{In memoriam: This paper is dedicated to the memory of my beloved friend and mentor, \\Ralph Brooks Turner, Ph.d. `Last of the free spirits.'}


\begin{abstract}
Let $\Gamma$ be a subgroup of $\PSL$ generated by three parabolic transformations. The main goal of this paper is to present an algorithm to determine whether or not $\Gamma$ is discrete. Historically discreteness algorithms have been considered within several broader mathematical paradigms: the \textit{discreteness problem}, the construction and deformation of \textit{hyperbolic structures} on surfaces and notions of \textit{automata} for groups. Each of these approaches yield equivalent results. The second goal of this paper is to give an exposition of the basic ideas needed to interpret these equivalences, emphasizing related works and future directions of inquiry.    
 
\end{abstract}
\maketitle








\section{Introduction}\label{sec:intro}

Let $\Gamma$ be a subgroup of $\PSL$ generated by three parabolic transformations. The main goal of this paper is to present an algorithm to determine whether or not $\Gamma$ is discrete. Historically discreteness algorithms have been considered within several broader mathematical paradigms: the \textit{discreteness problem}, the construction and deformation of \textit{hyperbolic structures} on surfaces and notions of \textit{automata} for groups. Each of these approaches yield equivalent results. The second goal of this paper is to give an exposition of the basic ideas needed to interpret these equivalences, emphasizing related works and future directions of inquiry.   

Discreteness is central to the theory of \textit{geometric structures on manifolds.} Given a topological manifold $M$, by a geometric structure on $M$, or a \emph{$(G,X)$-structure} for short, one understands $M$ as being \textit{locally modeled} by the Lie group $G$ on the locally homogeneous space $X$ associated to $G$~\cite{Thurston1}~\cite{Goldman4}. For example if $G$ is $\Isom(X)$, the \textit{isometry group} of $X$, the action of the Lie group $G$ on $X$ preserves a \textit{Riemannian metric} on $X$, and hence on $M$. The construction of (global) \textit{symmetric spaces} $X = G / K$ where $G$ is a classical Lie group and $K$ is a maximal compact subgroup of $G$ demonstrates globally symmetric spaces are examples of locally homogenous spaces~\cite{WitteMorris}. Specific examples prominently featured in our discussion here are the \textit{hyperbolic plane} $\HS \simeq \SL / \SO$ and \textit{hyperbolic 3-space} $\HS^3 \simeq \SLC / \SU$, along with their respective \textit{isometry groups} $\PSL$ and $\PSLC$. 

\begin{example}\label{ex:Fuchsian} 
Let $\Gamma <\PSL$, be a finitely generated non-elementary subgroup.  If $\Gamma$ acts on $\HS$ properly discontinuously and freely, then $\Gamma \backslash \HS$ is biholomorphic to a Riemann surface which admits a complete hyperbolic metric. Notice the use here of the Uniformization theorem, identifying the \textit{conformal structure} with \textit{hyperbolic structure}~\cite{Beardon}~\cite{Katok}. $\Gamma$ is an example of a \textit{Fuchsian} group.
\end{example}

\begin{example}\label{ex:Kleinian} 
Let $\Gamma < \PSLC$ be a finitely generated non-elementary subgroup. If $\Gamma$ acts on $\HS^3$ properly discontinuously and freely, then $\Gamma \backslash \HS^3$ is homeomorphic to the interior of a compact 3-manifold which admits a complete hyperbolic metric. Notice the analogous uniformization theorem here is W. Thurston's classification of hyperbolic 3-manifolds, in particular the version of \textit{the Geometrization Conjecture}\footnote{Thurston's revolutionary work first suggested an analogous theorem to Uniformization might be true in 3-dimensions~\cite{Thurston2}.} proved by Thurston~\cite{Thurston2}~\cite{Marden1}. $\Gamma$ is an example of a \textit{Kleinian} group.
\end{example}

For a fixed locally homogeneous space $X$ with isometry group $\Isom(X)$, the discreteness problem may be described as the creation of a \textit{procedure} which certifies whether or not a given subgroup $\Gamma$ of $\Isom(X)$ is discrete. A procedure which \textit{halts} in a finite number of steps is called an \textit{algorithm}. Discreteness algorithms were originally pursued in the context of Examples~\ref{ex:Fuchsian} and ~\ref{ex:Kleinian}~\cite{Stillwell1}~\cite{Riley}~\cite{GilmanMaskit}. Parametrizing $(G,X)-$structures admitted by some manifold $M$ leads to a very rich theory~\cite{BersandGardiner}~\cite{Bers3}~\cite{Canary1}~\cite{Goldman4}~\cite{Goldman5}~\cite{Labourie1}~\cite{Marden1}~\cite{Maskit2}\newline~\cite{Papadopoulos1}~\cite{Thurston1}~\cite{Thurston2}~\cite{Wolpert1}.  Discreteness --- within \textit{any} programme to investigate the manner in which a given manifold admits geometric structures and the interplay between these structures --- must be understood as central precisely by its exhibition of geometry via \textit{symmetry}. Herein we are primarily concerned with the particular case of \textit{hyperbolic structures} on surfaces. In particular, we understand $\Gamma$ as defined in \textit{Theorem~\ref{thm:MAIN1}} below as answering the following question: Does $\Gamma$ determine a \textit{complete hyperbolic structure} on the \textit{four-punctured sphere} $S_{0,4}$?

One strategy for understanding subgroups $\Gamma < \PSL$ is to construct good \textit{fundamental domains} for the action of $\Gamma$ on $\HS$. \textit{Reduction theory} (the theory of constructing ``good" fundamental domains for group actions) has a long history, from Gauss's work ``Reduction Theory of Quadratic Forms,"  to Borel--Harish-Chandra's ``Arithmetic subgroups of algebraic groups." As an example of this kind of phenomenon, it is a theorem of Siegel that $\Gamma \backslash \HS$ has finite area if and only if $\Gamma$ is \textit{finitely generated.} Whereby the existence of a finite sided \textit{fundamental domain} $\triangle$ ($\triangle$ is a convex polygon possibly with vertices on the boundary of $\HS$) for the $\Gamma$ action on $\HS$ corresponds to combinatorial-group-theoretic information, namely $\Gamma$ is \textit{finitely generated}.  However the analogue of this fact is not necessarily true for $\Gamma$ Kleinian. There are finitely generated Kleinian groups which are not \textit{geometrically finite}~\cite{Canary1}. 

Given an arbitrary \textit{non-elementary} subgroup $\Gamma < \PSL$, a necessary condition for discreteness is \textit{J\o rgensen's inequality}~\cite{Jorgensen}. A consequence of J\o rgensen's inequality is the following (rather intractable example of a) sufficient condition for discreteness: $\Gamma$ is discrete only if every \textit{non-elementary} \textit{rank 2} subgroup is discrete. Rank 2 means $\Gamma$  is generated by two elements, $\Gamma = \langle A,B\rangle$ for $A,B \in \PSL$.  The \textit{Poincar\'e polygon theorem} is an example of a more \textit{effective} sufficient condition, more effective in the sense that it holds for polygons $\triangle$ embedded in the hyperbolic plane $\HS$ with rational endpoints~\cite{Maskit1}. However, necessary and sufficient criteria for discreteness \textit{can not} exist, via the following work of Gilman-Maskit. 

\begin{theorem}\label{thm:GM} [Gilman-Maskit] For  $\Gamma < \PSL$, a non-elementary rank 2 subgroup, this discreteness problem was settled by J. Gilman \& B. Maskit~\cite{GilmanMaskit}. The solution to the discreteness problem is an algorithm, we will denote it by the $GM-$algorithm. Via analysis of the complexity of this algorithm, no necessary and sufficient condition for discreteness exists~\cite{GilmanMaskit}~\cite{Gilman97}.
\end{theorem}

 The $GM-$algorithm couples algebraic and geometric conceptual procedures, each one informing the other. Ultimately, the method of proof utilizes the intrinsic geometry of the hyperbolic plane to determine if a \textit{fundamental domain} $\triangle$, can be built for the $\Gamma$ action on $\HS$. The following first appeared in~\cite{Ashley}. 

\begin{theorem}\label{thm:MAIN1}[A.]
Let $A$, $B$ and $C$ be \textit{parabolic} elements of $\PSL$. For $\Gamma = \langle A,B,C\rangle$ non-elementary free subgroup, a procedure to determine whether or not $\Gamma$ is discrete exists. Furthermore, the procedure is an algorithm if $\Gamma$ has no elliptic elements, we will denote as the \emph{$4PS-$algorithm}, since $\Gamma$ corresponds to a four-punctured sphere group. 
\end{theorem}
\textsc{outline:} In Section 2 we present fundamental ideas of hyperbolic structures on surfaces. In Section 3 we summarize the $GM-$algorithm and notions of automata for groups. In Section 4 we give a self contained proof of Theorem~\ref{thm:MAIN1}. In Section 5 we conclude with brief interpretations and indications of future directions of inquiry. 

\specialsection*{Acknowledgements}
I am sincerely grateful to NAM and the AMS for the opportunity to be a part of this special 50th Anniversary proceedings. I would like to thank Todd Drumm and William Goldman for initiating me into this work while I was a graduate student at Howard University. I am also deeply indebted to Karen Smith and Richard Canary for mentoring me as a postdoctoral researcher at the University of Michigan.  I have also benefitted greatly from affiliation with the GEAR Network, MSRI, AMS-MRC, and NAM; I happily thank my colleagues and friends for their support and encouragement.  Finally, I thank each of the referees, for patiently reading earlier drafts of this paper and making many helpful comments and suggestions. 

\section{Background}

The purpose of this section is to habilitate the $4PS-$algorithm within the mathematical theory in which it lives. We begin with examples of constructions of Riemann surfaces, and proceed with deformations of geometric structures on surfaces, focusing on hyperbolic structures in the paradigm of locally homogenous structures. 

\begin{definition}\label{def:RS}A \textit{Riemann surface} $S$ is a 1-dimensional complex manifold. 
\end{definition}

As a consequence of Definition~\ref{def:RS}, Riemann surfaces are orientable, triangulizable smooth surfaces on which a certain collection complex valued continuous functions are designated as \textit{holomorphic}.  One way of defining such a \textit{conformal structure} on a given oriented smooth surface is to construct (via Gauss) a Riemannian metric and define \textit{local coordinate charts} as conformal homeomorphisms to $\C$~\cite{Bers3}.

\begin{definition}\label{def:hyperbolicstructure}
A Riemann surface $S$ is called \textit{hyperbolic}, if it admits a \textit{Riemannian metric} $ds_{\HS}$ of constant negative curvature,  where the boundary of $S$ (if nonempty) is totally geodesic~\cite{FarbMargalit}. 
\end{definition}


\subsection{Construction of Riemann surfaces}\label{sec:constructions}

We desire to distinguish between Riemann surfaces $S$ and the underlying topological surface $\Sigma$. We will denote \textit{closed} (compact, no boundary) Riemann surfaces of \textit{genus} $g$ by  $S_g$ and Riemann surfaces of \textit{finite type} (genus $g$, $n$ \textit{marked points}) by $S_{g,n}$; similarly  $\Sigma_{g}$ and $\Sigma_{g,n}$ (genus $g$, $n$ punctures) denote the underlying topological surfaces of $S_g$ and $S_{g,n}$, respectively. The fundamental group of any surface will be denoted by $\pi_1$.
 
 \begin{example}(Algebraic Construction)  Riemann surfaces\footnote{Riemann surfaces first appeared in the Ph.D thesis of Bernhard Riemann, [1851] ``Theory of functions of one complex variable."} were originally described as maximal domains on which holomorphic functions (themselves multi-valued via the process of \textit{analytic continuation} on domains $\Omega$ of $\C$) become single-valued. So described, Riemann surfaces are manifestly \textit{algebraic}, namely as $n$-sheeted \textit{branched coverings} of the complex projective line, $\CPS$.  The details of this construction yield a canonical method of \textit{compactification}, thus there is a natural correspondence between \textit{irreducible smooth projective algebraic curves} defined over $\C$ and \textit{compact} Riemann surfaces~\cite{Donaldson}~\cite{FarkasKra}~\cite{FrickeandKlein}.
\end{example}
 
 \begin{example}(Arithmetic Construction) A natural question regarding the algebraic construction is: when does a Riemann surface correspond to an irreducible algebraic curve defined over $\overline{\Q}$? Examples involve \textit{the Modular group} $\PSLZ$, \textit{congruence subgroups}, or any \textit{arithmetic subgroup} of $\PSL$~\cite{FrickeandKlein}~\cite{Reid}. 
 \end{example}

\begin{example}\label{exam:fundamentalgroup} (Topological Construction) The topological classification of surfaces gives a combinatorial representation of $\pi_1$ as a $(4g +n)-$gon with $2g$ side-pairing identifications:
\begin{equation}\label{fundamentalgroup}
\pi_1 = < a_1, b_1, \ldots , a_g, b_g, c_1, \ldots, c_n \ \big{|}
|  [a_1,b_1]  \cdots [a_g,b_g] \cdot c_1 \cdots c_n = 1 >. 
\end{equation}
$[a_i,b_i] = a_i\cdot b_i \cdot {a_i}^{-1} \cdot {b_i}^{-1}$ denotes the \textit{commutator} of the side-pairing generators $a_i$ and $b_i$ and the $c_i$ are boundary generators~\cite{Donaldson}.
The \textit{Euler characteristic}, denoted $\chi (\Sigma_{g,n})$ is a topological invariant and is given by $(2-2g)-n$. It is a classical result, that if $n \ge1$ then $\pi_1$ is isomorphic to a \textit{free group} of rank $N$, denoted $\F_N$, where $N= 2g +n +1$. For example,  with $\Gamma$ as in Theorem~\ref{thm:MAIN1}, $ \pi_1  \simeq \F_3$, a free product of three \textit{infinite-cyclic} groups~\cite{Stallings}.
 \end{example}

\begin{example}\label{exam:Uniformization}[Uniformization Theorem: Poincar\'e, Kobe, Klein] If $X$ is a simply connected Riemann surface, then $X$ is isomorphic to $\CPS$, $\CS$ or $\HS$~\cite{Springer}.
\end{example}

The isomorphism above is with respect to the Riemann surface structure. I.e. the \textit{uniformizing} isomorphism is a \textit{diffeomorphism} that respects the holomorphic structure, i.e. it is a \textit{holomorphic isomorphism}~\cite{Donaldson}. Since holomorphic maps are conformal, near every point of $S$ (locally) the metric alluded to in the remarks after Definition~\ref{def:RS} can be written as $ds = \lambda(z)|dz|$, where $z$ is the ``uniformizing" parameter, and an appropriate choice of \textit{conformal factor} $\lambda(z)$, can be made so $\kappa$ is $-1$. 

An equivalent statement of uniformization is: every connected Riemann surface $S$ is biholomorphic to $ \Gamma \backslash X$, where $X = \CPS, \CS,$ or $\HS$ and $\Gamma$ acts on $X$ \textit{properly discontinuously}, \textit{freely}, and by \textit{biholomorphisms}~\cite{Beardon}. This implies:
\begin{enumerate} 
\item $\pi_1 $ acts by isometries on $X$,
\item $\pi_1$ is isomorphic to  $\Gamma$, where $\Gamma$ is a discrete subgroup of $\Isom(X)$.
\end{enumerate}

The $\pi_1$ action is by isometries, since the conformal automorphisms of $X$ are precisely the isometries of $X$, i.e. $\Aut(X) =  \Isom (X)$.\footnote{This is often described as ``a miracle" of low-dimensional topology!} Uniformization is a powerful conceptual bridge, providing a correspondence between \textit{conformal structures} and \textit{metric structures} on a surface. For example, thinking of a Riemann surface  $\S_{g,n}$ with negative Euler characteristic as coming from a uniformization identifies $\pi_1$ isomorphically with a discrete subgroup $\Gamma$ of $\PSL$. Indeed if $\chi(\Sigma) <0$, the \textit{Gauss-Bonnet} theorem implies that  $\Sigma$ is a Riemann surface which admits a metric of constant negative curvature $\kappa$, hence its universal cover is $X = \HS$, and $\Isom(X) = \PSL$.    

\begin{remark} It is worthwhile to note that uniformization is a deep analytic result and is by no means a straightforward constructive process. Also, Examples~\ref{exam:fundamentalgroup} and~\ref{exam:Uniformization} imply that the Riemann surface detected by Theorem~\ref{thm:MAIN1} must be the four-punctured sphere, $\Sigma_{0,4}$. Since $\Gamma \simeq \F_3$ and $\Gamma$ acting freely on $\HS$ as in Theorem~\ref{thm:MAIN1} implies $\Gamma \backslash \HS \simeq S$ has four \textit{ends} which are \textit{cusps}. I.e. $\Gamma \backslash \HS \simeq S$ has four ends which are complete and have finite area. 
\end{remark}

 The following dynamical construction is a one of the first methods of building examples of Fuchsian and Kleinian groups~\cite{Beardon}~\cite{ChenANDTo}.

\begin{example}\label{exam:pingpong}(Ping-Pong Lemma) Let $A,B \in \PSL$ and let \newline ${U_A}^{-}, {U_A}^{+}, {U_B}^{-}, {U_B}^{+}$ be pairwise disjoint \textit{half-spaces} $\subset \HS$ such that ${U_A}^+$ is \textit{pairwise disjoint} with ${U_B}^-$ and ${U_A}^-$ is \textit{pairwise disjoint} with ${U_B}^+$. If  $A \cdot(\mathbb{H} \backslash \{\overline{U_A^{-}}\}) = {U_A}^+$ and $B\cdot(\mathbb{H} \backslash \{\overline{U_B^{-}}\}) = {U_B}^+$ then all of the following are true:

\begin{enumerate}
 \item $\langle A,B \rangle = \Gamma < \PSL$ is freely generated by $A$ and $B$,
 \item $\Gamma$ is discrete,
 \item $\Gamma$ acts freely on $\mathbb{H}$,
 \item $\forall \gamma \in \Gamma \backslash \{\Id\}$, $\gamma$ is hyperbolic,
 \item $\mathbb{H} \backslash \big\{ \{\overline{U_A^{-}}\})  \cup \{ U_A^{+} \})  \cup \{\overline{U_B^{-}} \}   \cup \{ U_B^{+} \}   \big\}$ is a fundamental domain for the $\Gamma$ action on $\mathbb{H}$.
 \end{enumerate}
\end{example}

The following geometric construction is a sufficient condition for discreteness and a main part of the conceptual framework of the $GM-$algorithm~\cite{Maskit1}~\cite{Beardon}.
\begin{example}[Poincar\'e Polygon Theorem]
For all $d \ge 5$, there exists a \textit{regular} right-angled $d-$gon $\triangle$ embedded in $\HS$, so that $\Gamma$ (the \textit{double-cover} of the subgroup generated by reflections along the sides of $\triangle$) is a discrete subgroup of $\PSL$ is discrete. Furthermore $\triangle$ is a fundamental domain for the $\Gamma$ action on $\HS$ and $S \simeq \Gamma \backslash \HS$ is closed or finite area hyperbolic surface. [Note that the Riemann surface constructed in this manner may be a hyperbolic \textit{orbifold}. I.e. $S \simeq \Gamma \backslash \HS$ may be an \textit{incomplete} hyperbolic structure, meaning the hyperbolic metric may have \textit{conical singularities}.]
\end{example}

\subsection{Hyperbolic Geometry}

The intrinsic geometry of the \textit{hyperbolic plane}\footnote{There are several models for the hyperbolic plane. Each model has particular advantages and disadvantages for understanding phenomena in the hyperbolic plane~\cite{Beardon}~\cite{Katok}.} may be described via tools of Riemannian geometry as a (real) 2-dimensional  \textit{homogeneous} space which admits a metric of constant negative \textit{curvature}~\cite{Beardon}~\cite{ChenANDTo}\newline~\cite{Katok}~\cite{Milnor1}~\cite{Marden1}.  A \textit{model} for the hyperbolic plane is a pair $(X , ds_{X})$, where $X$ denotes the underlying set and  $ds_{X}$ denotes the metric. The \textit{Upper Half Plane} model of the hyperbolic plane is: $\big( \HS = \{ x + iy \in \CS | \ y > 0 \}$ , ${ds_{\HS}}= \frac{dx^2 + dy^2}{y^2}\big)$. 

\begin{equation}\label{PSL}
\PSL := \SL / \{ \pm \mathbb{I} \} = \{\left[\begin{array}{cc} a & b \\  c & d \end{array}\right] : a,b,c,d \in \R, ad - bc = 1\}  / \{ \pm \mathbb{I} \}.  
\end{equation} 
$\PSL$  acts on $\HS$ via \textit{M\"{o}bius transformations}. That is, for any $z \in \HS$, and for any $A \in \PSL$ (equivalently $-A)$,  $A\cdot z \mapsto \frac{az+b}{cz+d}$. The metric $ds_{\HS}$ induces a notion of hyperbolic distance denoted $d_{\HS}(z,w)$ between any two points $z,w$ in $\HS$. By joining $z$ to $w$ with a smooth curve $\gamma$ in $\HS$, the \textit{hyperbolic length} denoted $\ell_{\HS}(\gamma)$ is given by $\int_{\gamma} |ds_{\HS}|$. Finally, the hyperbolic distance $d_{\HS}(z,w)$ is defined as $ \inf_{\gamma} \ell_{\HS}(\gamma)$. 

\begin{definition}\label{def:Isometry} A mapping $f: \HS \rightarrow \HS$ is an \textit{isometry} if it preserves the hyperbolic distance, that is $d_{\HS}(f(z),f(w)) = d_{\HS}(z,w)$, for all $z,w \in \HS.$
\end{definition}

Viz \`a viz the overarching principle of symmetry, isometries are objects of extreme utility, as are \textit{geodesics}. For any two points $x,y$ not on the boundary, there is a unique geodesic between them. $\HS$ is therefore a \textit{geodesic space}. The boundary of $\HS$, denoted by $\partial_{\infty}\HS$, is the \textit{real projective line} $\RPS := \R \cup \{ \infty \}$. $\PSL = \Isom^+(\HS)$, the \textit{orientation preserving isometries} of $\HS$. This follows from the fact that $\PSLC$ is the group of conformal automorphisms of the complex projective line $\CPS$, where the action is also by M\"{o}bius transformations but the coefficients are restricted from $\C$ to $\R$. The restriction of the action from $\PSLC$ to $\PSL$ is characterized by fixing a $\RPS$ inside $\CPS$, equivalently $\PSL$ preserves a $\HS \subset  \CPS$. [This gives an alternative description of the hyperbolic metric $ds_{\HS}$ as a conformally Euclidean metric.] Furthermore, identifying $\R^2$ with $\C$ and using the geometry of complex numbers, along with the fact that M\"{o}bius transformations are generated by an even number of compositions of reflections in circles and lines, all isometries of $\HS$ can be understood to be generated by (possibly finite compositions of) the following: \textit{reflection} about $y-$axis $(x,y) \mapsto (-x,y)$, \textit{horizontal translation} $(x,y) \mapsto (x+c,y)$, \textit{scaling} or \textit{dilatation} $(x,y) \mapsto (cx,cy)$, \textit{inversion in the unit circle} $(x, y) \mapsto (\frac{x}{x^2 + y^2}, \frac{y}{x^2 + y^2}$)~\cite{Marden1}. 

\begin{definition}
Submanifolds of (real) dimension one which minimize distance are called \textit{geodesics}. 
\end{definition}
Via horizontal shifting and scaling, all geodesics of $\HS$ are characterized as either vertical lines or circular arcs that intersect $\R$ orthogonally. Similarly isometries may be classified geometrically by their \textit{fixed points.} $A \in \PSL$ is called: \textit{elliptic} if $A$ has one fixed point in $\HS$, \textit{hyperbolic} if $A$ has two fixed points on $\partial_{\infty}\HS$, \textit{parabolic} if $A$ has \textit{exactly} one fixed point on $\partial_{\infty}\HS$. Irrespective of the model of the hyperbolic plane, geodesics will be fixed as sets (not point-wise but globally) of \textit{involutions.} It is worthwhile to observe that hyperbolic isometries may also be characterized by open half-panes, as in the ping-pong construction of Example~\ref{exam:pingpong}. In particular, \textit{disjoint pairwise orthogonal} geodesics determine disjoint \textit{open half-spaces} which give a fundamental domain for the action of a hyperbolic isometry.  Alternatively isometries are classified algebraically. $A \in \PSL$ as in Equation~\ref{PSL} is elliptic, hyperbolic, or parabolic if $|\Tr A| < 2, |\Tr A| > 2,$ or $|\Tr A| =2$, respectively where $|\Tr| := |a+d|$ denotes the absolute value of the trace of $A$. The algebraic characterizations of geodesics are independent of model of the hyperbolic plane~\cite{Beardon}~\cite{Marden1}.  

The topology induced on $\HS$ by either the hyperbolic or Euclidean metric is the same. Therefore the action of any $\Gamma < \PSL$ \textit{extends} to the \textit{Euclidian closure} of $\HS$, denoted $\overline{\HS} := \R \cup \{\infty\}$. For $z_0 \in \HS$, the set of all possible \textit{limit points} of $\Gamma$-orbits $\Gamma \cdot z_0$, is called the \textit{limit set} of $\Gamma$ and is denoted by $\Lambda (\Gamma)$. It can be shown that $\Lambda(\Gamma)$ does not depend on the point $z_0$. If $\Gamma$ is discrete then $\Lambda (\Gamma)$ is necessarily a subset of $\overline{\HS}$~\cite{Beardon}.

\begin{definition}
A discrete subgroup $\Gamma$ of $\PSL$ is called \textit{elementary} if it has a finite orbit in $\Lambda (\Gamma)$.  Otherwise, $\Gamma$ is called \textit{non-elementary.} 
\end{definition}
A non-elementary subgroup $\Gamma$ of $\PSL$ must have one hyperbolic element. In fact, any such $\Gamma$ has infinitely many such elements, no two of which have a common fixed point. Many algebraic, geometric, and dynamic consequences can be derived if a group is non-elementary~\cite{Beardon}~\cite{Katok}. 

\begin{definition}
Discrete non-elementary subgroups of $\PSL$ are called Fuchsian. Discrete non-elementary subgroups of $\PSLC$ are called Kleinian. 
\end{definition}

The following proposition is a necessary condition for discreteness and is a main part of the conceptual framework of $GM-$algorithm.  

\begin{prop}(J\o rgensen's inequality)~\cite{Jorgensen} If $\langle A,B \rangle$ is a non-elementary Kleinian group then:
\[
|\Tr(A)^2 - 4| + |\Tr([A,B]) - 2| \ge 1.
\]
\end{prop}



\subsection{Dynamics of Fuchsian Groups}

The \textit{covering theory} implicit in the examples in Section~\ref{sec:constructions} leads to a fundamental dynamical question: When and where does a group act \textit{properly discontinuously}? 

\begin{definition}\label{def:properdiscontinuously}
A group $\Gamma$ acts \textit{properly discontinuously} on a topological space $X$ if for every compact set $K \subseteq X$, $\{ \gamma \in \Gamma | \ \ \gamma \cdot K \cap K \ne \emptyset \}$ is finite.
\end{definition}

\begin{definition}For $\Gamma$ and $X$ as in Definition~\ref{def:properdiscontinuously}, a region $\Omega \subseteq X$ where $\Gamma$ acts properly discontinuously is called a \textit{domain of discontinuity}, denoted by $\Omega(\Gamma)$.
\end{definition}

Let $M$ be a topological manifold and $\Gamma$ be a group, $\Gamma$ is said to act on $M$ by \textit{orientation preserving homeomorphism} if and only if there exists a \textit{non-trivial} homomorphism $\rho : \Gamma \rightarrow \Homeo^+(M)$. Via the \textit{compact-open} topology, $\Homeo^+(M)$ is a topological group~\cite{ChenANDTo}. There is a natural bijective correspondence between actions of $\Gamma$ on $M$ and representations of $\Gamma$ into $\Sym(X)$~\cite{Shalen}. An action is said to be \textit{effective} if it corresponds to a \textit{faithful} representation~\cite{Shalen}. Along with a \textit{fixed-point} theorem which describes canonical neighborhoods of \textit{torsion} points, properly discontinuous actions give the following desired quotient construction.
\begin{prop}
If $X$ is a Hausdorff, locally compact topological space and if $\Gamma$ acts on $X$ by homeomorphisms and $\Gamma$ acts properly discontinuously and freely on $X$, then the quotient space $\Gamma \backslash X$ is a Hausdorff topological surface via the quotient topology~\cite{ChenANDTo}.
\end{prop}

We may characterize discreteness with the following equivalencies: a finitely generated subgroup $\Gamma$ of  $\PSL$ acts properly discontinuously on $\HS$ if and only if $\Gamma$ is a discrete subgroup of $\PSL$ if and only if there exists a finite-sided convex fundamental domain $\triangle$ for the $\Gamma$ action on $\HS$~\cite{Beardon}~\cite{Katok}. 

\begin{definition} Given a topological manifold $M$ and a group $\Gamma$ acting on $M$ properly discontinuously, a subset $\Omega \subseteq M$ is called a \textit{fundamental domain} if and only if each of the following hold:
\begin{enumerate}
\item $\overline{\Omega} \longrightarrow \Gamma \backslash M$ is surjective. I.e. every $\Gamma$-orbit meets $\overline{\Omega}$,
\item $\Int(\Omega) \longrightarrow \Gamma \backslash M$ is injective.  I.e. $\gamma \cdot \Int(\Omega)$ are disjoint for each $\gamma \in \Gamma$,
\item $ \partial \Omega$, the boundary of $\Omega$, defined by $ \partial \Omega := \overline{\Omega} \backslash \Int(\Omega)$, is ``small." I.e. for any $\Gamma$-\textit{invariant measure} on $\Omega$, the measure of the boundary is zero,
\item (Local Finiteness) $\forall \gamma \in \Gamma, \forall$ compact subsets $K \subset X$, $M$ meets only finitely many translates of $\gamma \cdot \overline{\Omega}$. 
\end{enumerate}
\end{definition}

\begin{definition}\label{Dirichlet} For any $z_0 \in \HS$ only fixed by $\Id \in \Gamma$, $D_{\Gamma}(z_0)$ the \textit{Dirichlet fundamental domain} is:  $D_{\Gamma}(z_0)  = \{ z \in \HS \ | \ \ d(z,z_0) \le d(z, \gamma \cdot z_0), \ \forall \gamma \ne \Id \in \Gamma\}.$
\end{definition}

\begin{remark} We now emphasize our focus on the paradigm of locally homogenous structures.  Propositions~\ref{prop:limitsetP1} and~\ref{prop:limitsetP1} demonstrate how it is possible to parametrize the limit set $\Lambda (\Gamma)$ of a hyperbolic structure with the boundary of the hyperbolic plane. Example~\ref{ex:QF} demonstrates how to further employ the limit set $\Lambda (\Gamma)$ to parametrize deformations of hyperbolic structures~\cite{Goldman4}~\cite{Labourie1}~\cite{Thurston1}.  [These should be understood as motivation for the manner in which a Riemann surface $S$ which admits a hyperbolic metric may be advantageously interpreted as a hyperbolic $(G,X)-$structure on $\Sigma$. Recall how $G$ and $X$ are described proceeding Example~\ref{ex:Fuchsian} in Section~\ref{sec:intro}; these hyperbolic geometric structures are characterized by two pieces of data: $\rho : \pi_1\Sigma \longrightarrow \PSL$ and $\Dev: \widetilde{\Sigma} \longrightarrow \mathbb{H}$.]
\end{remark}

\begin{prop}\label{prop:limitsetP1}
Given a representation $\rho : \pi_1\Sigma_g \longrightarrow \PSL$ such that there exists a mapping $\xi: \partial_{\infty}(\pi_1\Sigma_g) \longrightarrow \partial_{\infty}\HS$ continuous, injective, and $\rho-$\textit{equivariant} (I.e. $\xi$ \textit{intertwines the action} of $\pi_1S$ with action of $\rho$ on $\partial_{\infty}\HS$), then:
\begin{enumerate}
\item $\rho$ is injective, 
\item $\rho$ is discrete, 
\item $\rho(\pi_1\Sigma)$ is \textit{cocompact}, that is $\rho(\pi_1\Sigma)$ has compact quotient. 
\end{enumerate}
\end{prop}
As we have defined $\Gamma$ in Theorem~\ref{thm:MAIN1}, $\Sigma_{g,n} \simeq \Gamma \backslash \HS$ does not have boundary. Yet, it is worthwhile (regarding \textit{compactifications}) to consider the case of Riemann surfaces $\Sigma^r_{g,n}$, genus $g$, $n$ punctures and $r$ boundary components, that is $r$ \textit{ends} that are \textit{collar neighborhoods} of closed geodesic boundary elements. 
\begin{prop}\label{prop:limitsetP2}
Given a representation $\rho: \pi_1\Sigma^r_{g,n} \longrightarrow \PSL$ such that there exists a mapping  $\xi : \partial_{\infty}(\pi_1\Sigma^r_{g,n})  \longrightarrow \partial_{\infty} \HS$ continuous, injective and 
$\rho$-equivariant, then: 
\begin{enumerate}
\item $\rho$ is injective,
\item $\rho$ is discrete,
\item $\rho(\gamma)$ is hyperbolic,  $\forall \gamma \ne  \Id \in \Gamma$.
\end{enumerate}
\end{prop}

It is a fact that $\Dev$ is $\rho-$equivariant and we will see $\Dev$ is an isometry which extends to the boundary of the hyperbolic plane. Therefore $\Dev$ can be identified with $\xi$ in Propositions~\ref{prop:limitsetP1},~\ref{prop:limitsetP1}.  Furthermore, the limit set $\Lambda_{\Gamma}$ is the entire boundary of the hyperbolic plane, $\partial_{\infty}\HS$. [For surfaces with boundary, like $\Sigma^r_{g,n}$ in Proposition~\ref{prop:limitsetP2},  $\Lambda_{\Gamma}$ will be a \textit{Cantor set} contained in $\partial_{\infty}\HS$. Indeed,  $\partial_{\infty} \pi_1(\Sigma^r_{g,n}$)  is \textit{homotopy equivalent} to the \textit{Cayley graph} of $\pi_1(\Sigma^r_{g,n})]$. In both cases, $\Dev$ ``encodes" the entire representation $\rho$. For example the \textit{Hausdorff dimension} of  $\Lambda_{\Gamma}$ gives a notion of ``geometric magnitude" of a surface, and this is one sense in which a representation $\rho$ is encoded by the \textit{boundary map} $\Dev$~\cite{Labourie1}~\cite{Goldman4}. 

\begin{example}\label{ex:QF} (\textit{Quasi-Fuchsian} Structures) R. Riley developed one of the first computer programs for detecting discreteness~\cite{Riley}. Riley's procedure detected \textit{quasi-Fuchsian} structures $\mathcal{Q}\mathcal{F}(\Sigma)$. Given a Fuchsian structure on $\Sigma$ (a discrete and faithful representation $\rho : \pi_1 \longrightarrow \PSL$) a quasi-Fuchsian structure can be constructed by ``wiggling" the representation $\rho$ a ``bounded" amount so that the image of the limit set $\rho(\Lambda_{\Gamma})$ is deformed from $\RPS$ (Fuchsian) to a \textit{Jordan curve} (quasi-Fuchsian). [The precise way to affect this type of \textit{quasi-conformal} deformation is described via \textit{Beltrami differentials}~\cite{Marden1}.] It is a fact that $\Gamma \simeq \pi_1$ acts properly discontinuously and cocompactly on $\Omega(\Gamma_{\rho}) := \CPS \backslash \{ \xi(\partial_{\infty} \pi_1)\}$. Also $\xi: \partial_{\infty} \pi_1 \longrightarrow \partial_{\infty} \HS$, so that the domain of discontinuity $\Omega(\Gamma_{\rho})$ is homeomorphic to $\HS \cup \HS^{-}$, where $\HS^{-}$ denotes the \textit{Lower Half Plane}. This construction gives a uniformization of a quasi-Fuchsian manifold $M_{\rho} = \Gamma \backslash (\HS^3 \cup \Omega(\Gamma_{\rho}))$~\cite{Maskit2}~\cite{Marden1}.  
\end{example}
\begin{theorem}\label{thm:Riley}[Riley] Discreteness is \textit{decidable} for non-elementary finitely generated subgroups of $\SL$~\cite{Riley}~\cite{Kapovich2}. (See Section~\ref{sec:decidability} for relevant definitions of decidability models and halting sets.)
\end{theorem}

\subsection{Deformations of Hyperbolic Structures}
The \textit{deformation theory} of geometric structures is naturally involved in the discreteness problem; we have observed via example the correspondence between the classification of Riemann surfaces and discrete subgroups of the automorphism groups of $\CPS, \CS,$ or $\HS$.   A \textit{deformation space} of geometric structures is a space whose points themselves are geometric structures, whereby deformation spaces parametrize geometric structures admissible on a given topological manifold~\cite{Goldman4}. The \textit{Deformation theory} of geometric structures is quite vast. Two sources which emphasize the themes of algebra and analysis at the heart of complex function theory of Riemann surfaces are~\cite{mumford}and~\cite{Bers3}. For a series of surveys see~\cite{Papadopoulos1}. Also see~\cite{Canary1}\newline~\cite{ChenANDTo}~\cite{FarbMargalit}\cite{Goldman1}\cite{Goldman2}~\cite{Goldman3}~\cite{Goldman4}~\cite{Goldman5}~\cite{Labourie1}~\cite{Marden1}~\cite{Thurston1}\newline~\cite{Wentworth}~\cite{Wolpert1}. There are many different constructions of deformation spaces of Riemann surfaces: Mumford's algebraic construction of $\mathcal{M}_{g}$ (the moduli space of stable algebraic curves), Teichm{\"u}ller's complex-analytic construction of $\mathcal{T}(S_g)$ (the space of  conformal structures on $S_g$), Hitchin's complex gauge-theoretic construction of $\mathfrak{M}_{n,g}$ (the moduli space of rank$-n$ stable Higgs bundles on $S_g$.)  We will briefly describe Fenchel-Nielsen's real-analytic construction of the \textit{Fricke space} $\mathcal{F}(\Sigma_{g})$, the space of \textit{complete marked hyperbolic structures} on $\Sigma_{g}$. This construction emphasizes uniformization, and is based on work of Fricke-Klein which develops the deformation theory of hyperbolic structures on a surface in terms of the space of representations of its fundamental group $\pi_1$ into $\SLC$~\cite{FrickeandKlein}~\cite{Goldman1}.   


\begin{definition} A \textit{marked} hyperbolic structure is a pair $(S_1, \phi_1 )$, where $S$ and $S_1$ are Riemann surfaces and $\phi_1 : S \longrightarrow S_1$ is a diffeomorphism such that if there exists an isometry $I: S_1 \longrightarrow S_2$, then  $\phi_{1} \simeq I \circ \phi_{2}$  up to \textit{homotopy}~\cite{FarbMargalit}. 
\end{definition}

A marking defines an equivalence relation on hyperbolic structures~\cite{FarbMargalit}. Also, since any two Riemann surfaces of constant curvature -1 are locally isometric and since via uniformization any local isometry from a connected subdomain of $\HS$ \textit{uniquely} extends to an isometry on all of $\HS$, a marked structure determines two pieces of data: a map which globalizes the coordinate charts of $S$,  $\Dev : \widetilde{S} \longrightarrow X$ and a representation which globalizes the coordinate changes of $S$, $\rho : \pi_1 \longrightarrow G$. $\Dev$  is called \textit{developing map},  $\rho$ is called \textit{holonomy representation} and $X$ is the homogeneous space of $G$. This is precisely the data of a $(G,X)$-structure on $\Sigma$~\cite{Goldman4}.
 
 \begin{definition} The pair $(\Dev, \rho)$ is called a \textit{developing pair} and determines a $(G,X)$-structure on $\Sigma$. 
 \end{definition}

\begin{definition} The $\SL$-\textit{representation variety}, denoted $\mathcal{R}(\pi_1)$, is given by:  $\mathcal{R}(\pi_1) = \Hom(\pi_1, \SL) / \SL$, where $\SL$ acts by conjugation. 
\end{definition}
It is possible to construct $\mathcal{D}_{(G ,X)}(\Sigma)$, the deformation space of all $(G,X)$-structures on $\Sigma$~\cite{Goldman4}.  The points of such deformation spaces are isomorphism classes of $(G,X)$-structures on $\Sigma$. When $G = \PSLC$ these deformation spaces are called \textit{Fricke space} $\mathcal{F}(\Sigma_{g})$. Furthermore, $\Mod$ \textit{the mapping class group} of $\Sigma$, whose elements maybe be described as \textit{isotopy} classes of elements of ${\Homeo}^{+}(\Sigma$) which fix the boundary point-wise (if non-empty) and preserve the set of punctures, acts on  $\mathcal{D}_{(G,X)}(\Sigma)$ by changing the marking on the underlying surface~\cite{FarbMargalit}.

\begin{theorem}(Fricke) Let $G$ be $\SLC$, $\Mod$ acts properly discontinuously but not necessarily freely on $\mathcal{F}(S_g)$. 
\end{theorem}

The holonomy representation defines a mapping from the deformation space of $(G,X)$-structures to the G-\textit{representation variety}: 
  \begin{equation}\label{eq:holonomy}
  \rho : \mathcal{D}_{(G ,X)}(\Sigma) \longrightarrow \Hom(\pi_1,G)/G.
  \end{equation}

It can be shown that $\rho$ is $\Mod$-equivariant~\cite{Goldman4}~\cite{Labourie1}~\cite{Thurston1}.

\begin{theorem}(Thurston) In Equation~\ref{eq:holonomy}, $\rho$ is a local homeomorphism. Furthermore if $\Sigma$ be $\Sigma_g$, then $\rho$ is an embedding. (This is a theorem of A. Weil when $G = \PSL$, and is implicit in the work of C. Ehrasman)~\cite{Goldman4}.
 \end{theorem}

 \begin{remark} In $4PS-$algorithm,  $G = \PSL$, $\pi_1 \simeq \F_3$, and $\pi_1$ acts freely on $\HS$. This deformation space is the Fricke space $\mathcal{F}(\Sigma)$ contains equivalence classes of  \textit{complete} \textit{marked} hyperbolic structures on $\Sigma_{0,4}$.  
 \end{remark}
 
For closed surfaces the uniformization theorem can be reinterpreted in this context as identifying $\mathcal{F}(\Sigma_g)$ with $\mathcal{T}(S_{g})$~\cite{Goldman1}. In fact, the holonomy representation $\rho$ embeds $\mathcal{F}(\Sigma_{g})$ as a \textit{connected component} of $\Hom(\pi_1,\PSL)/\PSL$. It is a theorem of Goldman that representations $\rho$ with maximal \textit{Euler class} $e(\rho)$, (maximal with respect to the \textit{Milnor-Wood inequality}: $|e(\rho)| \le |\chi(\Sigma)|$) are discrete and faithful and correspond to an entire component of $\mathcal{R}(\pi_1)$~\cite{Goldman2}.  Note that $e(\rho)$ may be understood as a generalization of \textit{rotation number} of a homeomorphism. [Just as the unit tangent bundle on a surface corresponds to the space of directions, the fiber-wise restriction of a Riemannian metric on a vector bundle endows each fiber with an \textit{angular form} $\psi$. In particular, $e(\rho)$ is defined as the \textit{differential} of the angular form $d\psi$ in each fiber of a \textit{sphere-bundle} on $S$~\cite{Milnor2}~\cite{Goldman2}.]  Moreover Goldman shows every integer suggested by the Milnor-Wood occurs (there are $4g-3$ components of $\mathcal{R}(\pi_1)$ for $S_g$), and that the $\Mod$ action is not properly discontinuous on the non-maximal components of $\mathcal{R}(\pi_1)$~\cite{Goldman2}. The behavior of non-maximal components is more mysterious dynamically, but by no means less interesting, they contain points which correspond to \textit{singular} hyperbolic structures~\cite{Goldman3}.

The most natural way to understand the topology of $\mathcal{F}(\Sigma_{g})$ is via Fenchel-Nielsen coordinates, which are \textit{length-twist} coordinates on pairs of \textit{hyperbolic pants}. These homeomorphically identify $\mathcal{F}(\Sigma_{g})$ with $\R^{6g-6}$. In particular, the deformation space of complete hyperbolic structures on $\Sigma_g$ has real dimension $dim_{\R} =6g-6$. 


Certain \textit{algebro-geometric-analytic} objects called \textit{character varieties}, denoted here by $\mathfrak{X}_{\Gamma}(G)$, parametrize the deformation spaces of hyperbolic structures on $\Sigma$~\cite{Goldman1}~\cite{ABL}.  Character varieties are obtained from representation varieties by forming an appropriate \textit{geometric invariant theory} quotient. Precisely, we make the following identification of orbits: 
 \[
 \mathfrak{X}_{\Gamma}(G) := \Hom(\Gamma , G) / /  G, \text {where } [\rho_1] \sim [\rho_2] \text{ if and only if }  \overline{[\rho_1]} \cap \overline{[\rho_2]} \ne \emptyset.
\]
 A theorem of Procesi is  $\mathfrak{X}_{\Gamma}(G)$ is generated by trace functions, $|\Tr|$~\cite{Goldman1}~\cite{ABL}. 
 
 \begin{theorem}\label{thm:Goldman} [Goldman] In~\cite{Goldman1}, the $\SLC$-character variety of a rank 2 free $\F_2$ is identified with $\C^3$; furthermore, by relating the $FN$-coordinates of the \textit{three-holed sphere} $\Sigma_{0,3}$ to trace parameters corresponding to the boundary components, the Fricke space of $\Sigma_{0,3}$ is identified with $(-\infty, -2]^3$.  Goldman also gives a geometric description of the irreducible $\SLC$-representations of $\F_2$ as ``mildly degenerate" hexagons, from which he builds a \textit{Coxeter extension} of $\F_2$, via reflection in the common perpendiculars of the axes; ultimately describing the Fricke space of the \textit{one-holed torus} $\Sigma_{1,1}$. See~\cite{Goldman1} for details, including descriptions of the character varieties of $\Sigma_{0,4}$ and $\Sigma_{1,2}$. 
 \end{theorem}
 \begin{theorem}\label{thm:MPT}[Maloni-Palesi-Tan] In~\cite{MaloniPSPT}, using Morse theoretic techniques the $\Mod$ action on \textit{relative} $\SLC$-character varieties of $\Sigma_{0,4}$ is described.  In particular, the existence of a non-empty domain of discontinuity for $\PSL$ representations of 4-punctured sphere groups is demonstrated. Whereby regions of $\PSL$-character varieties of $\Sigma_{0,4}$ whose points correspond to singular hyperbolic metrics on $\Sigma_{0,4}$ are parametrized.
 \end{theorem}

\section{Automata}
Now we turn to notions of \textit{decidability} and \textit{complexity} of algorithms for groups. The comparative analysis and complexity of algorithms requires making a lot of choices~\cite{BSS}~\cite{Epstein}~\cite{Gilman95}~\cite{Gilman97}~\cite{Stillwell1}~\cite{Weinberger}~\cite{Kapovich1}. The choices may seem naive, but they are actually quite subtle. For example, what does one mean by algorithm? Or what does one mean by decidability? --- The complexity analysis of algorithms is more qualitative but choice remains. For example, complexity should measure what? We follow~\cite{GilmanMaskit}~\cite{Gilman97}, taking an \textit{algorithm} to be a procedure which stops in finite time on \textit{all} input values, as apposed to a \textit{semi-algorithm} which may go on forever or output nonsense for some input values. The concept of the $GM-$algorithm as a \textit{two-state automata} was first raised in~\cite{Silverio}, we do not pursue here.

\subsection{$GM-$algorithm} The $GM-$algorithm codifies all previous $2-$generator $\PSL$ discreteness algorithms. The overarching conceptual framework is to pair the process of \textit{trace minimizing} with the \textit{Poincar\'e/J\o rgensen dichotomy}~\cite{Gilman95}.  All possible types of generating pairs are explicated into two disjoint cases: \textit{non-intertwining}  and \textit{intertwining}, two hyperbolic generators with intersecting axes and all other cases, respectively. The Poincar\'e/J\o rgensen dichotomy describes an \textit{``if-else-if"} check of the necessary condition(J\o rgensen) against the sufficient condition(Poincar\'e), with each change of generating set. The algorithm ends when a fundamental domain is built or can be determined to exist. Trace minimization works like this: starting with the ordered generating set $\{ A, B \}$, replace this generating set with the generating set $\{B,A\}$  (called a  \emph{switch}), $\{A,B^{-1}\}$ (called a \emph{inversion}), or $\{A,AB\}$ (called a \emph{twist})~\cite{Gilman95}. These replacement moves are called \textit{Nielson transformations}, they are done to minimize the absolute value of the traces of the generators or to prepare the generators for minimization. Gilman and Maskit also give a geometric interpretation of the $GM-$algorithm: trace minimizes corresponds to minimizing the length of the longest side of a hyperbolic triangle determined by the intersecting axes of $A$ and $B$.  Furthermore a \textit{Coxeter extension} $\widetilde{\Gamma}= \langle A, B, AB \rangle$ of $\Gamma=\langle A,B \rangle$ is employed by the $GM$-algorithm; either both $\widetilde{\Gamma}$ and ${\Gamma}$ are discrete or both are not discrete.  Lastly, we emphasize that understanding the \textit{commutators} of the changes in generators along with the so called \textit{acute triangle} theorem enables the $GM-$algorithm to decide when elliptic elements will or will not have finite order. [It is a fact that hyperbolic area is independent of generating set of underlying edge identifications~\cite{Beardon}.] The $GM-$algorithm determines whether a fundamental domain for $\Gamma > \PSL$ can or can not be built, in finitely many steps.

\subsection{Decidability}\label{sec:decidability}
Recall the celebrated theorem of G{\"o}del's which may be interpreted as saying \textit{undecidable} questions exist~\cite{Weinberger}. Many questions in topology and geometry we would like to build an algorithm for have been encoded by group theory into \textit{combinatorial group theory} and \textit{geometric group theory} (or \textit{logic})~\cite{Stillwell1}. Psychologically, a finite presentation like the one in Equation~2.1, can be thought of as describing a group like an axiomatic system~\cite{Weinberger}. A group is defined by generators, subject to a given set of relations. For example, the \textit{Tietze} theorem states that two presentations define the same group if and only if there is a sequence of Nielsen moves that relate the two presentations; the \textit{Novikov--Boone} theorem states that there are finitely presented groups with an unsolvable \textit{word problem}~\cite{Weinberger}. The \textit{word problem} asks one to give an algorithm for deciding whether a given combination of generators represents the trivial element of the group~\cite{Stillwell1}. It can be shown that this is an implicit property of a group, in the sense that it does not depend on the way the group is presented. We follow~\cite{Kapovich1} and take decidability to mean an algorithm \textit{halts} either on $P$ or on $\neg P$. Decidability questions in general depend on formulation. For example, decidability questions involving real value inputs $\R$ can be decided with two (not necessarily congruent) models, \textit{BSS-computability} or \textit{bit-computability}~\cite{Kapovich1}. 

\subsection{Complexity}
The analysis of the complexity of algorithms reveals that decidability questions are not merely dialecticism or logical piddling. The following is proved in~\cite{Gilman95}.
\begin{theorem}\label{thm:complexity}
If $\Gamma$ is as described in Theorem~\ref{thm:GM}, except the coefficients are restricted from $\R$ to $\Q$, then the complexity of discreteness algorithm can been seen to be \textit{linear} in the number of generators. 
\end{theorem}
 As a consequence of Theorem~\ref{thm:complexity}, necessary and sufficient conditions for discreteness do not exist, therefore an algorithm is the best possible solution of the discreteness problem described by Theorem~\ref{thm:GM}~\cite{Gilman95}. The following is a compelling related result.
 
 \begin{theorem}\label{thm:Kapovich1}[Kapovich] Discreteness is $BSS$-undecidable for free, rank$-2$ subgroups of $\SLC$~\cite{Kapovich1}.
 \end{theorem}
The proof of Theorem~\ref{thm:Kapovich1} involves demonstrating that boundary of the \textit{Maskit slice} (the Maskit slice is the set of \textit{geometrically finite} and \textit{faithful} representations of once-punctured torus groups into $\SLC$)  is \textsc{not} a countable union of algebraic sets. In particular, the boundary of the Maskit slice is not BSS-computable. This ultimately requires \textit{the ending lamination conjecture}~\cite{Minsky}~\cite{BrockCanaryMinsky}. We add one more result before concluding these brief remarks on algorithms. Using \textit{course geometry} (in the sense of Gromov~\cite{Gromov}) the following related discreteness problem is being prepared.
 \begin{theorem}\label{thm:Kapovich2} [Kapovich] Given $\rho : \F_k \longrightarrow \SLC$ a non-elementary representation, discreteness and faithfulness is bit-computability decidable~\cite{Kapovich2}.  
\end{theorem}

\section{Proof of Main Theorem}
The goal of this section is to give a proof of Theorem~\ref{thm:MAIN1}. Recall, the $4PS$-algorithm is for subgroups $\Gamma$ of $\PSL$ generated by three parabolic generators: $\langle A,B,C\rangle = \Gamma   \simeq \F_3$, and $\Gamma$ acts freely on $\HS$. The strategy of proof for the $4PS$-algorithm is the same as that of the GM-algorithm. Namely the proof proceeds via Poincar\'e/J\o rgensen dichotomy and trace minimizing. Either a fundamental domain for $\Gamma$ can be constructed after a finite sequence of changes of the \textit{ordered} generating set $\{A, B, C\}$, or else one is able to determine that $\Gamma$ is not discrete. Figure~\ref{fig:ThreeParGood} is an example of a fundamental domain of a discrete rank 3 free subgroup of $\PSL$. 

\begin{figure}
\centerline{\includegraphics[width=3.0in]{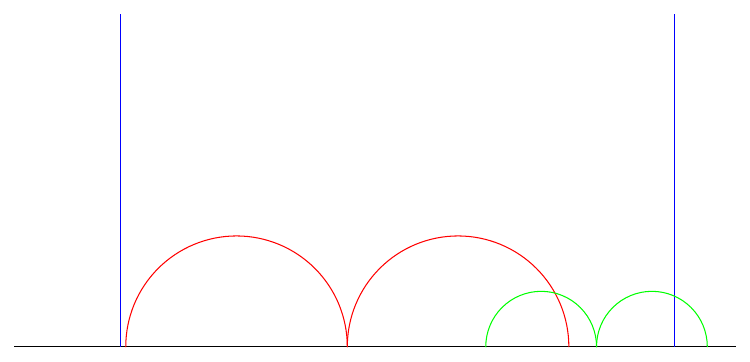}}
\caption{Configuration for stopping with 3 Parabolics}
\label{fig:ThreeParGood}
\end{figure}


\subsection{Set-Up}\label{sec:setup} As in the $2-$generator case,  conjugation of $\Gamma$ by any element $A \in \PGL$ leaves trace invariant. Hence it is possible to conjugate an element $A$ of the generating set, by an element of $\PGL$, say $X$, denoted by  $A \mapsto A^{X} := XAX^{-1}$ so that the fixed point of $A^{X}$ is $\infty$. Furthermore, we may conjugate $A^{X}$ by suitable diagonal matrix in $\PSL$ so that its translation length is $\pm 2$.  Next, we may conjugate the generating set by a parabolic which fixes $\infty$, so that the fixed point of one of the remaining generators, say $B$, is 0. Finally, we may conjugate the generating set by a reflection in the line from 0 to $\infty$ so that the fixed point of $C$ is a positive real, label this $x$. 

Passing to inverses if needed $A,B,C$ can be represented by the following matrices: 

\begin{equation}\label{eq:generators}
A=\begin{bmatrix}
1 & 2 \\  0&1 \end{bmatrix} , 
B=\begin{bmatrix}1 & 0 \\  \frac{-2}{y}&1 \end{bmatrix},
C= \begin{bmatrix}
1-\frac{2x}{z} & \frac{2x^2}{z} \\  \frac{-2}{z}& 1 + \frac{2x}{z} \end{bmatrix}.
\end{equation}

The significance of the triple of numbers $(x,y,z)$ lies in the geometry of the \textit{Ford domain}. 

\begin{definition} A \textit{Ford domain} for an element of $\PSL$ which does not fix $\infty$, is a ``limit" of a Dirichlet domain in which the center point $z_0$ goes to $\infty$. 
\end{definition}

\begin{figure}
\centerline{\includegraphics[width=3in]{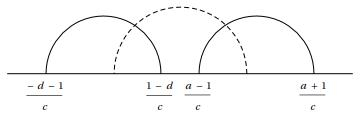}}
\caption{Ford domain for a generic hyperbolic element.}
\label{fig:HyperbolicFord}
\end{figure}
In particular, a Ford domain of a general hyperbolic element $G=$
$\begin{bmatrix}
a & b \\  c & d \end{bmatrix}$ is bounded by the geodesic whose endpoints are $\frac{-d}{c} \pm \frac{1}{c}$ and the geodesic $\frac{a}{c} \pm \frac{1}{c}$. Notice that the Ford domain is symmetric with respect to the vertical line centered at $\frac{a-d}{2c}$, see Figure~\ref{fig:HyperbolicFord}. The diameter of the compliment of one connected component of the complement of a Ford domain is $\frac{2}{|c|}$, which we will call the \textit{Ford strength} of the transformation. [Therefore, with $B$ and $C$ as in Equation~\ref{eq:generators}, $y$ is the Ford strength of $B$ and $z$ is the Ford strength of $C$. Also notice that the \textit{fixed points} of the hyperbolic transformation described in Figure~\ref{fig:HyperbolicFord} is not the midpoints of the compliment of the Ford domain.] We will call the Euclidean distance between the connected components of the complement of the Ford domain the \textit{inner Ford-distance}, given by $\frac{|\Tr G|-2}{|c|} $. We will call the diameter of the complement of the Ford domain the \textit{outer Ford-distance}, given by $\frac{|\Tr G| +2}{|c|}$.  

The following two Lemmas are necessary for the $4PS$-algorithm, see Step $\mathcal{D}$. Lemma~\ref{lem:FordStrength} gives criteria (involving the inner and outer F-distances) for determining when $A$ and a general non-elliptic generate a free discrete group with no elliptics.  Lemma~\ref{lem:testelliptics} gives a criterion (involving Ford strength) for determining if the product of some power of $A$ and a general non-elliptic element is itself elliptic.   

\begin{lemma}\label{lem:FordStrength}

 Let  $
 A =  \begin{bmatrix} 1 & 2 \\ 0 & 1 \end{bmatrix}$ and 
  $G = \begin{bmatrix} a & b \\ c & d \end{bmatrix}$ be any non-elliptic  element
  of $\PSL$. The following hold:
  \begin{enumerate}
  \item[(i)]  $AG, A^{-1}G, G^{-1}A$ and $G^{-1}A^{-1}$ are non-elliptic $\iff$ $|\Tr(G) - |2c|| \geq 2$,
 \item[(ii)] $\frac{|\Tr (G)| + 2}{|c|} < 2$, in which case $\langle A,G\rangle$ generate a free discrete group with no elliptics. 
\end{enumerate}
 \end{lemma}
 
 \begin{proof} 
\item[(i)] Since $G$ is not elliptic,  $|\Tr G| = |a+d| \geq 2$.  Without loss of generality, we can assume $a+d\geq 2$.
 The matrices $AG, A^{-1}G$, as well as $G^{-1}A$ and $G^{-1}A^{-1}$ have trace $(a+d) + \pm 2c$.   These elements are never elliptic when $\pm 2c\geq 0$. 
 So we only need to consider the case where the trace is $a+d - |2c|$, and $2c\neq 0$. 
These are elliptic if and only if $|(a+d) - |2c|| < 2$. 
\item[(ii)] A Ford domain for $\langle A,G\rangle$ is given by intersection of the Ford domain of $G$ and the half planes determined by $A$ centered at $\frac{a-d}{2c}$.  
\end{proof}

\begin{lemma}\label{lem:testelliptics}
Given $A,G$ as in Lemma~\ref{lem:testelliptics}, if $\frac{2}{|c|} >1$, then there exist some $n \in \mathbb{N}$ such that $A^nG$ is elliptic. 
\end{lemma}
Proof:
Observe that $A^nG = \begin{bmatrix}
a+2cn & b+2dn \\  c&d \end{bmatrix}$. Therefore, 
\[
\Tr(A^nG) = (a + d) + 2cn = \Tr(G) + 2cn.  
\]
And if $\frac{2}{|c|} > 1$ then $|2c|<4$. Hence there exists $n \in \mathbb{N}$ such that $|\Tr(A^nG)| <2.$ $\qed$ 


\subsection{Conjugation Calculations}

A subtlety of the $4PS-$algorithm is that Nielsen transformations do not necessarily affect trace minimization. 
\begin{example}\label{ex:minimallydirectedpairwise} 
Let $A,B,C$ be generating set. Consider making a combination of several Nielsen moves. For example, replace $B \mapsto ABA^{-1}$. Then the traces of the new generating set $\{A, B^A =ABA^{-1},C\}$ all remain unchanged, as expected. However, the trace of the product $B^{A}C$ \textit{will} change and in particular it will increase. 
\end{example}

Example~\ref{ex:minimallydirectedpairwise} demonstrates that contrary to the case for 2-generator groups, when Nielsen moves are applied to 3-generator groups the trace of generators \textit{may increase} after successive Nielsen moves. This motivates the following definition.

\begin{definition} For general $A$ and $B$, the pair $\{A,B\}$ is \textit{coherently oriented pairs} if and only if the $|\Tr(AB)|$ is less than $|\Tr(A^{-1}B)|$.  Similarly, the triple $\{A,B,C\}$ is called \textit{coherently oriented pairwise} if and only if they are coherently oriented pairs for each pair~\cite{GilmanKeen}.
\end{definition}
Geometrically, $\{A,B,C\}$ coherently oriented pairwise implies that trace minimization affects ping-pong, as was done in the dynamical construction or Riemann surface in Example~\ref{exam:pingpong}.

\begin{lemma}\label{lem:conjA}
Let $P$ be a parabolic which does not fix $\infty$, and let $A$ be in Equation~\ref{eq:generators}. The conjugation of $P$ by $A^n$ does not change the Ford strength of the product $P^{A^n}$ but translates the fixed point by $2n$. \end{lemma}

\begin{proof} This follows from a straightforward matrix multiplication. The proof is omitted. 
\end{proof}
It is worthwhile to notice the effect of a few specific conjugations on a few special configurations of fixed points $(x,y,z)$.  The following Lemma is necessary for the $4PS$-algorithm, see Step $\mathcal{S}$.

\begin{lemma}\label{lem:conjB}(Preparing for Minimization) Let  $A,B,C$ be as in Equation~\ref{eq:generators}. 
\begin{enumerate}
\item If $y \ne 2x$, then the fixed point of $C^{B}$ is $\frac{y^2z}{(y-2x)^2}$ and the distance between the fixed points of $B$ and $C^B$ is $\frac{xy}{|y-2x|}$.

\item If $y = 2x$ then the fixed point of $C^B$ is $\infty$. Hence either $\langle A,B,C\rangle$ is a $2-$generator subgroup or $\langle A,B,C\rangle$ is not discrete. 
\item If $x < y- \epsilon$, for some fixed $\epsilon > 0$.
\[
z \mapsto \frac{y^2z}{(y - 2(y-\epsilon))^2}  \ \ = \ \ \frac{y^2z}{(-y + 2 \epsilon)^2} \ \ = \ \ \frac{y^2z}{(y-2\epsilon)^2} \ \ = \ \ \frac{1}{(1 - \frac{2\epsilon}{y})^2}\cdot z
\]
Furthermore,
\[
x \mapsto \frac{yx}{y -2(y-\epsilon)|} \ \ = \ \ \frac{y}{|y - 2\epsilon|}\cdot x
\]
\end{enumerate}
\end{lemma}

\begin{proof}These calculations follow from the symmetry of the configuration of three parabolic elements, like those in Equation~\ref{eq:generators}. Conjugating $B$ by $A$ changes $x$ but not $y$ or $z$, hence it changes $\Tr(BC)$, where $ABA^{-1}$ is the replaces $B$ as a generator. Also it is possible to conjugate $B$ by $A$ so that $x < 1$. Conjugating $C$ by $B$ changes $x$ (we can assume $x$ is positive) but some conjugation will change the order of the fixed point of $B$ and the fixed point of $C$. This conjugation will also change $z$, replacing generators after conjugations $z$ will be $\frac{y^2z}{(y-2x)^2}$. In particular,  if $x = y$ then nothing changes. But if $ x< y$ then the new $z$ values is increased.
\end{proof}
One last configuration needs to be analyzed, see Step $\mathcal{D}$.
\begin{lemma}\label{fundamentaldomain}
Let $z \le y \le x \le 1+ \epsilon$. Also assume $\frac{x^2 }{|(2x - y -z)|} < 1$. Then for $A,B,C$ as in Equation~\ref{eq:generators} $\langle A,B, C \rangle = \Gamma$ is discrete. 
\end{lemma}
Proof:  It suffices to construct a Ford domain for $\Gamma = \langle A,B, C \rangle$. According to the assumptions we have $p = \frac{( 2x^2 -xz )}{|(2x - y -z)|}$ and $p = x \cdot \frac{( 2x -z )}{|(2x - y -z)|} > x.$ Therefore  
$C(p) = \frac{( xy )}{y + z} \ < \ x$. Hence we have demonstrated a region in $\HS$ bounded below by the geodesic whose endpoint are $p$ and $x$ and the geodesic whose endpoints are $x$ and $C(p)$.  Now notice that $B(C(p))$ is $\frac{( -xy )}{y + z}$. Hence the outer Ford distance of BC is less than 2 and a fundamental domain for $\langle A,B,C\rangle$ exists. %
\qed


\subsection{Proof of Theorem~\ref{thm:MAIN1}}
We begin with a summary of the conceptual strategy of the proof of Theorem~\ref{thm:MAIN1}. We begin by choosing a canonical starting configuration in the sense of Section~\ref{sec:setup}. Chosen in this manner, $A$ fixes $\infty$, $B$ fixes the origin and the fixed point of $C$ is $x$ apart from fixed point of $B$.  This configuration of parabolic-parabolic-parabolic generators of $\Gamma$ is \textit{canonical} up to values triples $(x,y,z)$ which correspond to lengths between fixed points. Nielsen transformations of coherently oriented pairwise ordered generating sets always decreases trace by some amount bounded away from zero and the distance between Ford strengths decreases as the trace of the product of a coherently oriented pair decreases. Via the configurations of fixed points under Nielsen moves we can decide if we can build a Ford domain or not. The sequential procedure of the algorithm is given by the following steps: \begin{equation}\label{eq:SADE}
\mathcal{S} \longrightarrow \mathcal{A} \longrightarrow \mathcal{D} \longrightarrow \mathcal{E}.\end{equation} 

\underline{STEP $\mathcal{S}$}: (Set-Up)
\begin{enumerate}
\item Fix real numbers $\epsilon >0$ and $\delta>0$.
\item If $x>1+\epsilon$, conjugate $C$ by a power of $A$ so that the fixed point of $B$ and the fixed point of $C$ are \textit{close}; that is, within 1+ $\epsilon$. $\{ A,B,C \} \mapsto \{A,B, C^{A^n}\}$. Notice that $x$ is decreasing by at least $2 \epsilon$, therefore the trace is seen to decrease by a fixed amount, bounded below by a function of $\epsilon$ on the order of ${\epsilon}^2$.

\item If $x < y - \delta$ conjugate $C$ by a power of $B$ so that the fixed point of $B$ and the fixed point of $C^B$ and fixed point of $A$ are \textit{close}. Here $x$ is decreasing by an amount which is bounded below by order $\delta$ by Lemma~\ref{lem:conjB}. Also $z$ is decreasing. In particular, $z \mapsto \frac{(y^2 z )}{(y - 2 x)^2}$. So notice that if $x = y$ then nothing changes. But if $x < y$  then the new $z$ value is decreased by an amount bounded below by order of ${\delta}.^2$ Most important, $\Tr(AC)$ decreases by an amount bounded below by the order of $\delta ^2$. 
\end{enumerate}

\underline{STEP $\mathcal{A}$}: (Algorithm Begins)\\ $\Tr(A) = \Tr(B) = \Tr(C) =2$. The $(2,1)-$element of $A$ is $0$, the $(2,1)-$element of $B$ is $\frac{-2}{y}$, and the $(2,1)-$element of $C$ is $\frac{-2}{z}$.
\begin{enumerate}
\item Calculate the matrix products; $A\cdot B, A\cdot C, B\cdot C$.
\item $\Tr(BC) = -2 \frac{4x^2}{yz}$, if hyperbolic $|\Tr(BC)| =  \frac{4x^2}{yz} - 2$. \textit{In particular}, $BC$ is hyperbolic if and only if $x^2 > yz$.
\item $\Tr(AB) = -2 \frac{4}{y}$, if hyperbolic $|\Tr(AB)| =  \frac{4}{y} -2$. 
\item $\Tr(AC) = -2 \frac{4}{z}$, if hyperbolic $|\Tr(AC)| =  \frac{4}{y} -2$.
\end{enumerate}

\underline{STEP $\mathcal{D}$}: (Decisions) \\By minimal positioning, the configuration is $x \ge y - \delta$ and $x \ge z - \delta$.
\begin{enumerate}
\item If $BC$ is elliptic the algorithm ends, Lemma~\ref{lem:FordStrength}.  If $BC$ not elliptic, then {\bf$x^2 \ge yz$}. Hence $x > \min{(y,z)}$.
\item If  $|(2x- y-z)|< yz$ then there is some $n$ such that $A^n BC$ is elliptic by Lemma~\ref{lem:testelliptics}  Therefore, {\bf{$|(2x- y-z)| \ge yz$}}. 
\item If $z < y <x <1$ then $1-x \le 1- y$ and $1-x \le 1-z$. In which case $\frac{|( x^2 -yz)|}{(2x - y -z)} <1$ that is the inner ford strength of BC is less than 2. In which case then either $ABC$ is elliptic or it is possible to create a fundamental  domain, hence $\Gamma$ is discrete, see Lemma~\ref{fundamentaldomain}.
\end{enumerate}

\underline{STEP $\mathcal{E}$}: (End Algorithm) \\
Now two possibilities remain observing the boundary configuration of the triple $(x,y,z)$. Without loss of generality we assume $y>z$.
\begin{enumerate}
\item  If after conjugations $z < y < x<1$ then all conditions in STEP $\mathcal{D}3$ are satisfied, hence either $ABC$ is elliptic or it is possible to create a fundamental domain.
\item  If after conjugations  $x>1$ and $x< 1+\epsilon$ then $2x-y-z> 0$, therefore the following holds: $\frac{( x^2 -yz)}{|(2x - y -z)|} \  = \  \frac{( x^2 -yz)}{(2x - y -z)} \  = \ \frac{( x^2 -yz)}{x - y + x - z} \  < \  \frac{( x^2 -yz)}{ (x - z)} \  <  \ \frac{( xy -yz)}{ (x - z)}  \ < \ 1$. Therefore this is Step $\mathcal{D}2$.
\item If  $(2x- y-z)< yz$ then there is some $n$ such that $A^n B C$ is elliptic, STEP $\mathcal{D}2$.
\item If $(2x- y-z)> yz$ then  $\frac{( x^2 -yz)}{(2x - y -z)}  < \frac{x^2}{xy} <1$ and by STEP $\mathcal{D}3$, $ABC$ is elliptic or it is possible to create a fundamental domain and $\Gamma$ is discrete.
\item If after conjugations $y>x$ and  $x > y - \delta$ then conjugate $C$ by $B$ and either  $z < y < x<1$  and this is STEP $\mathcal{E}1$. And is decidable. Or $x>1$ and $x< 1+\epsilon$ and is decidable too. 
\item If after conjugations $x < z - \delta$ then conjugate $B$  by $C$ and either condition in STEP $\mathcal{D}2$ decidable or STEP $\mathcal{E}2$. If so then conjugate $C$ by $B$ and decidable.
\end{enumerate}
\qed

\section{Interpretations}

We have presented the $4PS-$algorithm as a part of a sequence of related results: Theorem~\ref{thm:Riley}, Theorem~\ref{thm:GM}, Theorem~\ref{thm:Goldman}, Theorem~\ref{thm:MPT}, Theorem~\ref{thm:Kapovich1}, Theorem~\ref{thm:Kapovich2}. And each of these is circumscribed by the theory of $(G,X)$-structures on manifolds. Now we comment on why we began with this end in view. 

Aside from contextualizing the $4PS-$algorithm, examples and constructions of $(G,X)-$structures give indications of the tools employed in the proof of such results. [Of course the techniques and methods of proof are often more exotic, still the fundamental ideas persist, or if not, a tincture remains in generalizations.] By way of comparing and contrasting we emphasize the following reoccurring themes: the intrinsic geometry of the hyperbolic plane or hyperbolic 3-space (e.g. the Poincar\'e/J\o rgensen dichotomy, Coxeter extensions), utilizing informative words and dynamical techniques on character varieties (e.g. commutator words, relative character varieties), determining coordinates on character varieties (e.g. Goldman's coordinates on the $\SLC-$character varieties of rank 2 subgroups.)

Furthermore we desire to give an indication of directions of current related research.  At this time, a two-generator $\PSLC$ discreteness algorithm does not exists, but J. Gilman and L. Keen are working on this problem~\cite{GilmanKeen}. Also in the context of decision problems, discreteness can be replaced by some other property of discrete groups, for example $\Gamma$  \textit{convex cocompact} or $\Gamma$ \textit{Anosov}, and investigated algorithmically~\cite{Kapovich1}. No complete ``field guide" of $(G,X)-$structures exists.\footnote{The extraordinary book \underline{Indra's Pearls} may be the prototype of such a guide. It offers stunning computer visualizations of many of the rarities in the realm of Kleinian groups~\cite{Indra}.} By way of appealing to the mathematical precedence of advantageously reformulating problems, we emphasize that there are many correspondences for $(G,X)-$structures: $G-$representation varieties, $G-$local systems, $G-$Higgs bundles~\cite{Goldman4}~\cite{Goldman5}~\cite{Labourie1}. These corresponding theories offer up even more tools and techniques, which are more refined and hopefully more useful for proving theorems for discreteness algorithms. --- This phenomenon is illustrated quite remarkably by our discussion here of discreteness algorithms and related results for discreteness.
 
The quiddity of algorithms is that they are both everything and nothing. The most satisfactory solution to a discreteness problem might be an algorithm based on a computable system of coordinates for geometric structures, however searching for good coordinates might find us out to be nurturing chimera. [For example: the goal of \textit{Jacobi Inversion problem} in classical Abel-Jacobi theory was a complete analogue of Weierstrass's $\wp$-function, the goal of the \textit{Schottky Problem} is to compute the holomorphic invariants required to be able to distinguish holomorphic deformations of a principally polarized abelian variety.] It may not be possible to reformulate the former. An algorithm for the later exists, but the invariants are not computable.  
Guiding questions for creating discreteness algorithms remain: In which framework and by which means can we build \textit{computable coordinates} for geometric structures? 
\begin{question}
The geometric origins of \textit{cluster algebras} for surfaces with marked points is hyperbolic geometry and representation theory. \textit{Cluster coordinates} of type $\mathcal{A}_n$ correspond to triangulation of an $(n-3)-$gons. Can we use \textit{positivity} to detect discreteness, in particular discrete and faithful representations via cluster coordinates on simple hyperbolic surfaces~\cite{CalebGregKaren}?
\end{question}

\bibliography{mybib}
\bibliographystyle{alpha}

\end{document}